    \documentclass[12pt]{amsart}
 \usepackage{amsmath,amsthm,amsfonts,amssymb}

 \usepackage{mathrsfs}
\let\mathcal\mathscr

\def\ZZ{{\mathbf Z}}

\def\QQ{{\mathbf Q}}
\def\CC{{\mathbf C}}
\def\PP{{ \mathbf{P}}}

\def\OO{{\mathcal O}}
\def\R{\mathbf{R}}

\newcommand{\Numm}{N}

\newcommand{\tn}[1]{\textnormal{#1}}
\newcommand{\noi}{\noindent}

\newcommand{\RR}{\mathbf{R}}

\newcommand{\cO}{\mathcal{O}}

\newcommand{\bull}{_{\bullet}}

\newcommand{\eps}{\varepsilon}

\newcommand{\HH}[3]{H^{{#1}} \big( {#2} , {#3}
\big) }

\newcommand{\pr}{\prime}

\def\Z{{\bf Z}}

\def\C{{\bf C}}

\def\R{{\bf R}}
\def\Q{{\bf Q}}

\def\av{abelian variety}

\def\ppav{principally polarized abelian variety}
\def\pps{principally polarized abelian surface}

\def\cO{\mathcal{O}}

\def\cP{\mathcal{P}}

\def\cC{\mathcal{C}}
\def\cM{\mathcal{M}}

\def\lra{\longrightarrow}
\def\llra{\hbox to 10mm{\rightarrowfill}}
\def\lllra{\hbox to 15mm{\rightarrowfill}}

\def\llla{\hbox to 10mm{\leftarrowfill}}
\def\lllla{\hbox to 15mm{\leftarrowfill}}

\def\isom{\simeq}
\def\eps{\varepsilon}
\def\ie{\hbox{i.e.}}

\DeclareMathOperator{\Sym}{{\mathbf S}}

\DeclareMathOperator{\Pic}{Pic}

\def\Re{\mathop{\rm Re}\nolimits}

\DeclareMathOperator{\Peff}{Psef}
\DeclareMathOperator{\Psef}{Psef}
\DeclareMathOperator{\Nef}{Nef}
\DeclareMathOperator{\Pos}{Weak}
\DeclareMathOperator{\Weak}{Weak}
\DeclareMathOperator{\Spos}{Strong}
\DeclareMathOperator{\Stpo}{Strong}
\DeclareMathOperator{\Strong}{Strong}
\DeclareMathOperator{\Sepo}{Semi}
\DeclareMathOperator{\Sdef}{Semi}
\DeclareMathOperator{\Semi}{Semi}

\DeclareMathOperator{\GL}{GL}

\DeclareMathOperator{\End}{End}

\DeclareMathOperator{\Bl}{Bl}
\DeclareMathOperator{\PD}{PD}

\def\llra{\hbox to 10mm{\rightarrowfill}}
\def\lllra{\hbox to 15mm{\rightarrowfill}}

\newtheorem{lemm}{Lemma}[section]
\newtheorem{theo}[lemm]{Theorem}
\newtheorem{coro}[lemm]{Corollary}
\newtheorem{prop}[lemm]{Proposition}

\theoremstyle{definition}

\newtheorem{rema}[lemm]{Remark}
\newtheorem{remark}[lemm]{Remark}

\newtheorem{defi}[lemm]{Definition}
\newtheorem{exam}[lemm]{Example}
\newtheorem{example}[lemm]{Example}

\theoremstyle{remark}
\newtheorem*{remark*}{Remark}
\newtheorem*{note*}{Note}

 \theoremstyle{plain}

\newtheorem{theoremalpha}{Theorem}
\newtheorem{corollaryalpha}[theoremalpha]{Corollary}

\newtheorem{proposition/example}[lemm]{Proposition/Example}
\newtheorem{proposition}[lemm]{Proposition}
\newtheorem{corollary}[lemm]{Corollary}
\newtheorem{lemma}[lemm]{Lemma}
\newtheorem{problem}[lemm]{Problem}

\newtheorem{conjecture}[lemm]{Conjecture}
\newtheorem{theorem}[lemm]{Theorem}

\numberwithin{equation}{section}

\title{Pseudoeffective and nef classes on  abelian varieties}

\author{Olivier Debarre}
\address{\'Ecole Normale Sup\'{e}rieure\\
D\'{e}partement de Math\'{e}matiques et Applications\\
UMR CNRS 8553\\
45 rue d'Ulm\\
75230 Paris cedex 05, France}
\email{odebarre@dma.ens.fr}

\author{Lawrence Ein}
\address{Department of Mathematics, University of Illinois at Chicago, 851 S. Morgan Street, Chicago, IL 60607, USA}
\email{ein@math.uic.edu}
\thanks{Research of the second author partially supported by NSF grant DMS-0700774.}

\author{Robert Lazarsfeld}
 \address{Department of Mathematics, University of Michigan, Ann Arbor, MI
   48109, USA}
 \email{{\tt rlaz@umich.edu}}
  \thanks{Research of the third author partially supported by NSF grant DMS-0652845.}

\author{Claire Voisin}
\address{Institut de Math\'{e}matiques de Jussieu, 175 rue du Chevaleret, 75013 Paris, France}
\email{voisin@math.jussieu.fr}

\begin{document}

\maketitle

 \section*{Introduction}

 The cones of divisors and curves  defined by various positivity conditions on a smooth projective variety have been the subject of a great deal of work in algebraic geometry, and by now they are quite well understood.  However  the analogous cones for cycles of higher codimension and dimension have started to come into focus only recently, for instance in \cite{pet} and \cite{voi}. The purpose of this paper is to explore some of the phenomena that can occur   by working out the picture fairly completely   in a couple of simple but non-trivial cases. Specifically, we study cycles of arbitrary codimension on the self-product of an elliptic curve  with complex multiplication, as well as two dimensional cycles on the   product  of a very general abelian surface with itself.  Already one finds various non-classical behavior,  e.g.\ nef  cycles whose product is negative.\footnote{This answers a question raised in 1964  by Grothendieck \cite{Groth} in   correspondence with Mumford: see Remark \ref{Grothendieck.Questions}.} We also present a number of conjectures and problems for further investigation.

 Turning to more details, let $X$ be a smooth complex projective variety of dimension $n$. Given $0\le k\le n$, denote by $\Numm^k(X)$ the finite dimensional real vector space of numerical equivalence classes of codimension $k$ algebraic cycles on $X$ with real coefficients. We consider the closed convex cone
 \[
 \Peff^k(X) \ \subseteq \  \Numm^k(X)\]
generated by effective cycles. By analogy with the case of divisors, the elements of this cone are called pseudoeffective classes.  The vector space $\Numm^k(X)$ is dual to $\Numm^{n-k}(X)$, and (again extending the codimension one terminology) we define
 \[
 \Nef^k(X) \ \subseteq \ \Numm^k(X) \] to be the closed convex cone dual to $\Peff^{n-k}(X)$. Thus a class $\alpha \in \Numm^k(X)$ is nef if and only if $\big( \alpha \cdot \beta \big ) \ge 0$ for all effective cycles $\beta$ of dimension $k$.

 Now suppose that $B$ is an abelian variety: write $B = V/ \Lambda$ where $V$ is a complex vector space and $\Lambda \subset V$ is a lattice. Numerical and homological equivalence coincide on $B$ (\cite{lie}, or \cite[Theorem 4.11.1]{bl}), and therefore
  \[ \Numm^k(B) \ \subseteq \ H^{k,k}(B) \, \cap \, H^{2k}(B, \RR).
 \]
  So  elements of  $\Numm^k(B)$ are represented by real $(k,k)$-forms on $V$, and this leads to several further notions of positivity. Specifically, following the discussion in \cite[III.1]{dem} we say that a $(k,k)$-form $\eta$ on $V$ is \textit{strongly positive} if it is a non-negative real linear combination of forms of the type
$$  i\ell_1\wedge \overline \ell_1\wedge\cdots \wedge i \ell_k \wedge \overline \ell_k$$
for $\ell_j \in V^*$, and $\eta$ is \textit{weakly positive} if it restricts to a non-negative multiple of the canonical orientation form on any $k$-dimensional complex subspace $W \subseteq V$.      A $(k,k)$-form on $V$ is \textit{semipositive}  if it is real and if the associated Hermitian form  on $\bigwedge^kV$ is semipositive. This gives rise to a chain of three closed convex  cones:
\[
  \Strong^k(V)\ \subseteq \ \Semi^k(V)\ \subseteq \ \Weak^k(V) \]
sitting inside the space of real $(k,k)$-forms on $V$. When $k = 1$ or $k = n-1$ they coincide, but when $2 \le k \le n-2$ the inclusions are strict. Finally, one defines cones
\[
 \Strong^k(B)\ \subseteq \ \Semi^k(B)\ \subseteq \ \Weak^k(B)\]
in $\Numm^k(B)$, consisting of  those classes   represented by forms of the indicated type. One has
\[
\Psef^k(B) \ \subseteq \ \Strong^k(B) \ \ , \ \ \Weak^k(B) \ \subseteq \ \Nef^k(B),\] and in the classical case $k = 1$ of divisors these various flavors of positivity are actually all the same, i.e.,
\[
\Psef^1(B) \, = \, \Strong^1(B) \,= \, \Semi^1(B)\,  = \, \Weak^1(B) \,= \, \Nef^1(B)
\]
(with a similar statement for $k = n-1$).

Our first result computes the pseudoeffective and nef cones on the self-product  of an elliptic curve with complex multiplication:
\begin{theoremalpha} \label{Ell.Curve.Intro}
Let $E$ be an elliptic curve having complex multiplication, and set \[
B \ = \ E \times \ldots \times E \ \ \ (n\ \text{times}).\] Then for every $0 \le k \le n$:\[ \Psef^k(B) \,  = \,\Strong^k(B) \,  = \, \Strong^k(V) \ \ , \ \  \Nef^k(B) \, = \, \Weak^k(B) \, = \,\Weak^k(V). \]
 \end{theoremalpha}
\noi Here $V$ denotes as above the vector space of which $B$ is a quotient. It follows that    $B$ carries nef classes of every codimension $2 \le k \le n-2$ that are not pseudoeffective.  This implies formally the existence of nef classes whose product is not nef.

In the situation of the previous theorem,  the pseudoeffective and nef cones were described just in terms of positivity of forms. Our second computation shows that in general the picture can be more complicated:\begin{theoremalpha} \label{TheoremAIntro}
Let $A$ be a very general principally polarized abelian surface, and let $B = A \times A$. Then
$$\Peff^2(B) \, = \, \Spos^2(B) \, = \, \Sdef^2(B)\,  \subsetneqq \, \Pos^2(B) \, \subsetneqq  \ \Nef^2(B).$$
Furthermore,
$$\Peff^2(B) \, = \, \Sym^2\Peff^1(B), $$
where $\Sym^2\Peff^1(B)\subseteq \Numm^2(B)$ denotes the closed convex cone generated by products of elements of $\Peff^1(B)$.
\end{theoremalpha}

\noi Along the way to the theorem, we exhibit inequalities defining the pseudoeffective and nef cones. The statement of the theorem    remains valid for an arbitrary abelian surface $A$ provided that one intersects each term with the subspace of $\Numm^2(A \times A)$ generated by the products of certain natural divisor classes: when $A$ is very general this ``canonical subspace" fills out all of  $\Numm^2(A \times A)$. By the same token, it is enough to assume in the Theorem that $A$ is a very general abelian surface with a given polarization, or for that matter that $B$ is isogeneous to the product $A \times A$ appearing in the statement. By a specialization argument, the Theorem also implies:
\begin{corollaryalpha}
Let $A$ be a very general principally polarized abelian variety of dimension $g$. Then
\[ Semi^2(A \times A) \ = \ \Psef^2(A \times A) \ = \ \Sym^2 \Psef^1(A\times A)
\] and
\[  \Strong^{2g-2}(A \times A) \ = \ \Psef^{2g-2}(A \times A) \ = \ \Sym^{2g-2}\Psef^1 (A \times A) . \]
\end{corollaryalpha}

Theorem \ref{Ell.Curve.Intro}
 follows easily from the remark that $\Psef^1(B) = \Strong^1(V)$. By contrast, the argument leading to the second result is more computational in nature, the main point being to show that $\Sym^2\Peff^1(A \times A) = \Semi^2(A \times A)$. For this we exploit a natural $\GL_2(\R)$ action on the cones in question, and apply a classical argument in convexity theory. It would be interesting to have a more conceptual approach.

In the final section of the paper, we propose some conjectures and questions dealing with  positivity conditions on cycles of higher codimension. We hope that these may stimulate further work involving this  relatively uncharted circle of ideas.

Concerning the organization of this paper, we start in \S 1 with some general remarks about positivity of cycles on an abelian variety. The second section takes up the self-product of an elliptic curve with complex multiplication. In \S3 we begin the study of the product of an abelian variety with itself, and introduce there the algebra of ``canonical" classes. The main computations appear in \S4, while in \S 5 we give some complements. Finally, we propose  in \S 6 a considerable number of questions and open problems concerning positivity of higher codimension cycles.

 \medskip\noindent{\bf Acknowledgements.}
We have profited from many conversations with T. Peternell concerning positivity of higher codimension cycles. We've also benefited from discussions with  D. Edidin, W. Fulton, D. Maclagen, Y. Mustopa, and Y. Tschinkel concerning some of the questions in Section \ref{quco}.  Many thanks also to J.-B. Lasserre for providing the reference \cite{bar}, and to K. Ribet and B. Moonen for their explanations about Proposition \ref{tank}.  Finally, this research started during the program ``Algebraic Geometry'' at M.S.R.I. in 2009, and the authors thank this institution for support and for  excellent working conditions.

  \section{Positive classes on an abelian variety}
  \label{Pos.Classes.Abl.Var.Section}

 This section is devoted to some generalities about positivity of cycles on an abelian variety. After reviewing some facts about different notions of positivity for forms on a complex vector space, we introduce the basic cones that arise on abelian varieties. We conclude by analyzing them in the case of curves and divisors.

\subsection*{Positivity of forms on a complex vector space} We start by recalling some facts about positivity of $(k,k)$-forms on a complex vector space, following
 Chapter III.1 of Demailly's notes \cite{dem}.

 Let $V$ be a complex vector space of dimension $n$. If $(z_1,\dots,z_n)$ are
 complex coordinates on  $V$, then the  underlying real vector space is canonically oriented by the real $(n,n)$-form
$$ idz_1\wedge d\overline z_1\wedge\cdots \wedge  idz_n\wedge d\overline z_n\ =\  2^n dx_1\wedge dy_1\wedge\cdots \wedge  dx_n\wedge dy_n.$$
We denote by $\Lambda_{\RR}^{(k,k)}V^*$ the (real) vector space of real $(k,k)$-forms on $V$.
\begin{defi}\label{def1}
\begin{enumerate}
\item[(i).] A $(k,k)$-form $\eta$  on $V$ is {\em strongly positive} if it is a linear combination with nonnegative real coefficients of forms of the type
$$  i\ell_1\wedge \overline \ell_1\wedge\cdots \wedge i \ell_k \wedge \overline \ell_k $$
for linear forms $\ell_j \in V^*$.
\vskip 5pt
\item[(ii).] A $(k,k)$-form    $\eta$  is {\em weakly positive}\footnote{We are modifying somewhat the terminology in \cite{dem}.}
 if  the $(n,n)$-form
$\eta \wedge \omega$
is a real non-negative multiple of the orientation form for all strongly positive $(n-k, n-k)$-forms $\omega$. \end{enumerate}
\end{defi}
 \noi
 Strongly and weakly positive forms are real (\cite[III.1.5]{dem}). They form  closed convex cones (with non-empty interiors) in $\Lambda_{\RR}^{(k,k)}V^*$, which we denote by $\Strong^k(V)$ and $\Weak^k(V)$ respectively. Evidently $\Strong^k(V) \subseteq \Weak^k(V) $, and by construction there is a duality of cones:
 $$\Pos^k(V)=\Spos^{n-k}(V)^\vee.$$

\begin{remark}
It follows from the definition that
\[  \Strong^k(V) \ = \ \Sym^k \Strong^1(V), \]
where $\Sym^k \Strong^1(V)$ denotes the closed convex cone generated by  products of positive $(1,1)$-forms.
\end{remark}

\begin{remark}\label{Remark.Pos.Forms}
A $(k,k)$-form is weakly  positive if and only if it restricts to any $k$-dimensional complex vector subspace of $V$ as a nonnegative volume form (\cite[III.1.6]{dem}).
\end{remark}

  \begin{defi}
 A $(k,k)$-form   on $V$ is {\em semipositive} if it is real and the associated Hermitian form  on $\bigwedge^kV$ is semipositive.
 These forms form a real convex cone in $\Lambda_{\RR}^{(k,k)}V^*$ that  we denote by $\Sepo^k(V)$.
\end{defi}

 Using diagonalization, we see that the cone $\Semi^k(V)$ is generated by the forms $i^{k^2}\alpha\wedge\overline\alpha$, for $\alpha\in\bigwedge^kV^*$. It contains the cone $\Spos^k(V)$  and is self-dual, hence we have a chain of inclusions
    \begin{equation}\tag{$\cC_k$}
      \Strong^k(V)\ \subseteq \ \Semi^k(V)\ \subseteq \ \Weak^k(V)
          \end{equation}
whose dual is $(\cC_{n-k})$. If $\alpha\in\bigwedge^kV^*$, the semipositive form $i^{k^2}\alpha\wedge\overline\alpha$ is strongly positive if and only if $\alpha$ is decomposable (\cite[III.1.10]{dem}). Therefore, the inclusions above are strict for
 $2\le k\le n-2$.

 \subsection*{Classes on an abelian variety}
  We now turn to cohomology classes on an abelian variety.

 Let $B=V/\Lambda$ be an \av\ of dimension $n$. As in the Introduction, denote by
 $\Numm^k(B)$  the real vector subspace of  $\HH{2k}{B}{\RR}$  generated by classes of algebraic cycles. Thanks to \cite{lie} this coincides with the group of numerical equivalence classes of cycles. In the natural way we identify cohomology classes on $B$ as being given by $(k,k)$-forms on $V$, and we set $$\Spos^k(B) \ = \ \Spos^k(V)\cap N^k(B).$$ The closed convex cones
 \[ \Sdef^k(B) \ , \  \Pos^k(B) \ \subseteq \ \Numm^k(B)\]
are defined similarly. Thus
\[
\Strong^k(B) \ \subseteq \ \Semi^k(B) \ \subseteq \ \Weak^k(B).
\]
On the other hand, one defines as in the Introduction the cones
\[ \Psef^k(B) \ , \ \Nef^k(B) \ \subseteq \ \Numm^k(B) \]
of pseudoeffective and nef classes. Occasionally it will be convenient to work with cycles indexed by dimension rather than codimension. As customary, we indicate this by replacing superscripts by subscripts. Thus
\[
\Numm_k(B) \ =_{\text{def}} \ \Numm^{n-k}(B)  \ \ , \ \ \Nef_k(B) \ =_{\text{def}}\ \Nef^{n-k}(B),
\]
and so on.

As $B$ is homogeneous, the intersection of two pseudoeffective classes is again pseudoeffective, and in particular
\[  \Psef^k(B) \ \subseteq \ \Nef^k(B). \]
The following lemma refines this statement:
\begin{lemma} \label{Nef.Psef.Forms}
One has the inclusions:
\[
\Psef^k(B)\  \subseteq \ \Strong^k(B)  \ \subseteq \  \Weak^k(B) \ \subseteq \ \Nef^k(B).
\]
\end{lemma}

\begin{proof}  Let $Z \subseteq B$ be an irreducible subvariety of dimension $c$, and let $\omega \in \Weak^c(B)$ be a weakly positive $(c,c)$-form. Then $\int_Z \omega \ge 0$ thanks to Remark \ref{Remark.Pos.Forms}. Therefore any weakly positive class is nef. On the other hand, if $c = n-k$ and $\eta_Z$ is a $(k,k)$-form on $V$ representing the cohomology class of $Z$, then
\[   \int_B \eta_Z \wedge \beta \ = \ \int_Z \beta \]
for every $(n-k, n-k)$-form $\beta$. When $\beta \in \Weak^{n-k}(V)$ the integral in question is non-negative, and hence
\[
\eta_Z \, \in \, \Weak^{n-k}(V)^{\vee} \ = \ \Strong^k(V),
\]
 as required. \end{proof}

 Finally, we note that the various cones in question are preserved by isogenies:

 \begin{proposition} \label{Isogeny.Prop}
  Let $\phi : B^\pr \lra B$ be an isogeny. Then $\phi$ induces an isomorphism
  \[  \  \Numm^k(B) \overset{\cong}\lra \Numm^k(B^\pr) \]
  under which each of the five cones just considered for $B$ maps onto the corresponding cone  for $B^\pr$.  \end{proposition}
  \begin{proof}
  For the pseudoeffective and nef cones this follows from the projection formula, while for the other cones defined by positivity of $(k,k)$-forms it follows from the fact that the vector spaces underlying $B$ and $B^\pr$ are isomorphic.
  \end{proof}

  \subsection*{Duality} For later use we make a few remarks concerning duality.
    Denote by  $\widehat{B}=\Pic^0(B)$   the dual abelian variety of the \av\ $B$. For each $\ell\in\{0,\dots,2n\}$, there
are canonical isomorphisms
\begin{equation} H^{2n-\ell}(B,\Z)\stackrel{\PD}{\stackrel{{}_\sim}{\lra}} H_\ell(B,\Z)\stackrel{d}{\stackrel{{}_\sim}{\lra}}H^\ell(\widehat{B},\Z), \tag{*} \end{equation}
where the first isomorphism is   Poincar\'{e} duality on $B$. As for the second one, we
start from the canonical isomorphism $H_1(B,\Z)\stackrel{\sim}{\lra} H^1(\widehat B,\Z)$
defining $\widehat{B}$ and use the isomorphisms
$$H_\ell(B,\Z)\stackrel{{}_\sim}{\lra} \bigwedge^\ell H_1(B,\Z)\ ,\ H^\ell(\widehat{B},\Z)\stackrel{{}_\sim}{\lra}\bigwedge^\ell H^1(\widehat{B},\Z).$$
The isomorphisms in (*) are sometimes known as  the Fourier-Mukai transform.

\begin{prop} \label{pont}The   isomorphisms $d$
have the following properties:
\begin{itemize}
\item[(a).]  They exchange the Pontryagin product on $B$ and  the cup-product on $\widehat{B}$.
\vskip 5pt
\item[(b).]   They are compatible with the Hodge decompositions. When $\ell= 2k$ is even, they carry classes of algebraic cycles on $B$ to classes of algebraic cycles on $\widehat B$, and therefore define isomorphisms
 \[ N^{n-k}(B)\stackrel{{}_\sim}{\lra} N^k(\widehat{B}). \]
\item[(c).]  This isomorphism preserves
        the strongly positive cones.
\end{itemize}
\end{prop}

\begin{proof}[Proof of Proposition]
The Pontryagin product on  the homology of
$B$ is induced by the sum map $\sigma:B\times B\rightarrow B$. Its iterations provide the canonical isomorphisms
$\bigwedge^iH_1(B,\Z)\stackrel{{}_\sim}{\lra} H_i(B,\Z)$. Similarly, the cup-product on the cohomology of
$\widehat{B}$ provides the isomorphisms $\bigwedge^iH^1(\widehat{B},\Z)\stackrel{{}_\sim}{\lra} H^i(\widehat{B},\Z)$.
 Hence   (a) follows from the definition of $d$.

On $H_1(B,\Z)$, the map $d$ is given by interior product with the Hodge class
    $$c_1(\mathcal{P})\in H^1(B,\Z )\otimes H^1(\widehat{B},\Z )\subseteq H^2(B\times \widehat{B},\Z )$$
    of the Poincar\'{e} line bundle $\cP$. Therefore $d$ acts on $H_\ell(B , \ZZ)$  via cup product with a multiple of $c_1(\mathcal{P})^\ell$, which implies the assertions in statement (b).

For $k=1$ or $k=n-1$, item (c)   follows
  from the explicit representation of $d$ as induced by interior product
  with $c_1(\cP)$, and by choosing an explicit representative of $c_1(\cP)$ as in
  (\ref{Choose.Coords}) below. For general $k$, we observe that $\Strong^{n-k}(V)$
  is the convex cone generated by cup  products of $n-k$ elements of $\Strong^1(V)$  of rank $2$. But the image under $d\circ PD$ of such a decomposable class  is the
  Pontryagin product of $n-k$ elements of $\Strong^{n-1}(\widehat{V})$, hence is strongly positive.
\end{proof}

 \subsection*{Divisors and one-cycles}  We conclude this section by describing the cones in question in the classical cases of  divisors and curves. As above $B$ denotes an abelian variety of dimension $n$.

 To begin with, nef and pseudoeffective divisors coincide on any homogeneous variety. Therefore   Lemma \ref{Nef.Psef.Forms} yields the (well-known) equalities:
\begin{equation} \label{Cones.For.Divisors}
\Psef^1(B) \, = \, \Strong^1(B) \,= \, \Semi^1(B)\,  = \, \Weak^1(B) \,= \, \Nef^1(B) .
\end{equation}
 It follows dually that $\Nef^{n-1}(B) = \Psef^{n-1}(B)$, which implies similarly the analogue of \eqref{Cones.For.Divisors} for the cones of curves.

The next proposition asserts that  any pseudoeffective curve class can be written as a positive $\R$-linear combination of   intersections of pseudoeffective divisor classes.

 \begin{proposition}\label{prop17}
 Let $B$ be an abelian variety of dimension $n$.
 One has
 \[
\Psef^{n-1}(B) \ = \ \Sym^{n-1}   \Psef^1( B),
 \]
 where $\Sym^{n-1} \Psef^1(B) \subseteq \Numm^{n-1}(B)$ is the closed  convex cone generated by cup products of  pseudoeffective divisor classes.
  \end{proposition}

The proof will use the following Lemma, which involves the Pontryagin self-products of a curve class on $B$: given $\gamma \in H_2(B, \RR) = H^{2n-2}(B, \RR)$, we write $\gamma^{* (k)} \in H^{2n - 2k}(B, \RR)$ for the $k$-fold Pontryagin product of $\gamma$ with itself.

 \begin{lemm}\label{prodform}
 Let $B$ be an abelian variety of dimension $n$ and let $k$ be an integer, with $0\le k\le n$.
\begin{itemize}
\item[(a).] For any $\alpha\in H^2(B,\R)$, one has
 $$(\alpha^{n-1})^{*(n-k)}=(n-k)!(n-1)!^{n-k}\Big(\frac{\alpha ^n}{n!}\Big)^{n-k-1} \alpha^k .$$
\item[(b).] For any $\beta\in H^{2n-2}(B,\R)$, one has
 $$(\beta^{*(n-1)})^{n-k}=(n-k)!(n-1)!^{n-k}\Big(\frac{\beta^{*(n)}}{n!}\Big)^{n-k-1}\beta^k
 .$$
 \end{itemize}
\end{lemm}

\begin{proof}When $\alpha$ is an ample class, ``Poincar\'e's Formula'' (\cite{bl}, 16.5.6) reads
 $$\frac{\alpha^k}{k!}=\frac{d}{(n-k)!}\left( \frac{\alpha^{n-1}}{d(n-1)!}\right)^{*(n-k)} $$
where $d=\deg(\alpha)=\alpha^n/n!$. Since the ample cone has non-empty interior in $ N^1(B)$, this implies the equality in (a) for all classes $\alpha$ in $ N^1(B)$ (which is the only case where we will use it). But this equality can also be checked by representing any class in $H^2(B,\R)$ by a skew-symmetric form on $V$, very much as in \cite[4.10]{bl}, and it is easily seen that the forms represented by either side of the equality are proportional, with a proportionality constant depending only on $n$ and $k$. This constant must then be the one we just obtained.

Item (b) follows from (a) and Proposition \ref{pont} (a). \end{proof}

 \begin{proof}[Proof of Proposition]
 The issue is to show that $\Psef^{n-1}(B) \subseteq \Sym^{n-1}   \Psef^1( B)$.  To this end, consider a cohomology class    $\beta \in N^{2n-2}(B, \R)$ which   lies in the interior of $\Psef^{n-1}(B)$.
Then $\beta $ can be represented as a positive $\RR$-linear combination of an effective curve and a complete intersection of very ample divisors; in particular, $\beta $ generates $B$ and $\beta ^{*n}$ is non-zero. The effective divisor class $\beta^{*(n-1)} $ is then ample, and we are done by the formula (b) in the lemma.
  \end{proof}


\section{Products of CM elliptic curves}

In this section we consider cycles of arbitrary codimension on the self-product of an elliptic curve with complex multiplication. In this case the global cones coincide with those defined by linear algebra.

Let $E = \CC / \Gamma$ be an elliptic curve admitting complex multiplication, and put
$$   S \ =  \ \big \{s \in \CC \mid s\cdot \Gamma \subseteq  \Gamma \big \} \ = \ \tn{End}(E).$$
Thus $S$ is an order in an imaginary  quadratic extension of $\QQ$. We view the elements of $S$ interchangeably   as complex numbers or as  endomorphisms of $E$. Note that the $\RR$-span of $S$, seen in the first light, is all of
$\CC$.

Denote by $B$ the $n$-fold product $E^{\times n}$, which we write as usual $B = V / \Lambda$. We establish
\begin{theorem} \label{CM}  One has
\[  \Psef^k(B) \ = \ \Strong^k(B) \ = \ \Strong^k(V) \ \ , \ \ \Nef^k(B) \ = \ \Weak^k(B) \ = \ \Weak^k(V). \]
\end{theorem}
\noi The proof appears at the end of the section. First we record a corollary concerning the product of nef classes.

Specifically, we have seen that  when $2 \le k \le n-2$ there is a strict inclusion
\[  \Strong^k(V) \ \subsetneqq \ \Weak^k(V). \]
Therefore:
\begin{corollary}
For any $2 \le k \le n$, $B$ carries nef cycles of codimension $k$ that are not pseudoeffective. In particular, the product of nef cycles is not in general nef. \qed
\end{corollary}

\begin{remark}[Grothendieck's questions] \label{Grothendieck.Questions}
The corollary answers in the negative some questions raised by Grothendieck in 1964 in correspondence with Mumford \cite{Groth}. In this letter, Grothendieck starts by proposing some conjectures (subsequently settled by Mumford and Kleiman \cite{kle}) concerning numerical characterizations of amplitude for divisors. He goes on to write:

\small
\begin{quote}
I would like even a lot more to be true, namely the existence of a numerical theory of ampleness for cycles of any dimension. Assume for simplicity $X$ projective non singular connected of dim. $n$, let $A^i(X)$ be the vector space over $\QQ$ deduced from numerical equivalence for cycles of codimension $i$ (presumably this is of finite dimension over $\QQ$), and $A_i(X) = A^{n-i}(X)$ defined by cycles of dimension $i$, presumably $A_i$ and $A^i$ are dual to each other. Let $A_i^+$ be the cone generated by positive\footnote{i.e., effective.} cycles, and let $P^i \subset A^i$ be the polar cone. The elements of $P^i$ might be called pseudo-ample, and those in the interior of $P^i$ ample (which for $i = 1$ would check with the notion of ample divisor, if for instance the strengthening of Mumford-Nakai's conjecture considered above is valid\footnote{Earlier in the letter, Grothendieck had asked  whether it is enough to test positivity against curves in Nakai's criterion: Mumford's celebrated counterexample appears in his reply to Grothendieck. Recall however that  it is the content of Kleiman's work \cite{kle} that   the interior of $P^1$ is in fact the cone of ample divisor classes.}). The strongest in this direction I would like to conjecture is that the intersection of pseudo-ample (resp. ample) cycles is again pseudo-ample (ample), thus the intersection defines
\[  P^i \times P^j \lra P^{i+j}. \]
If $i$ and $j$ are complementary, $i + j =n$, this also means that the natural map $u_i: A^i \lra A_{n-i}$ maps $P^i$ into $A^+_{n-i}$ (and one certainly expects an ample cycle to be at least equivalent to a positive one!). For $i$ and $j$ arbitrary, the above inclusion can also be interpreted as meaning that the intersection of an ample cycle with a positive cycle is again (equivalent to) a positive cycle. Of course, one would expect an ample positive cycle to move a lot within its equivalence class, allowing to consider proper intersections with another given positive cycles. I wonder if you have any material against, or in favor of, these conjectures?
\end{quote}
\normalsize
Needless to say,  in the years since this letter was written  it has become abundantly clear that intuition from divisors is often a poor guide for higher codimensions. The computations of the present paper give yet another illustration of this principle.
\end{remark}

Finally, we give the
\begin{proof}[Proof of Theorem \ref{CM}] It suffices to show
\begin{equation}
 \Strong^1(V) \ \subseteq \ \Psef^1(B). \tag{*} \end{equation}
Indeed,   this implies that
\[
\Strong^k(V) \ = \ \Sym^k  \Strong^1(V)  \ \subseteq \ \Psef^k(B)
\]
thanks to the fact that the product of pseudoeffective classes on an abelian variety is pseudoeffective. The reverse inclusion being automatic, we obtain $\Strong^k(V) = \Psef^k(B)$, and hence (dually)  $\Nef^k(B) = \Weak^k(V)$.

For (*), write $B = \CC^{\times n}/ \Gamma^{\times n}$, and denote by $z_i$ the coordinate function on the
$i$-th component of $\CC^{\times n}$. Thus we may view
$(dz_1, \dots, dz_n)$ as a basis for the complex vector space $W $ of holomorphic $1$-forms on $B$. Consider the subgroup $M \subseteq W$ given by
$$   M = \big \{ \, s_1dz_1+ \dots + s_ndz_n \mid  \ s_i\in S\,  \big \}. $$
Note that any holomorphic one-form in $W$ is a non-negative $\RR$-linear combination of  elements in  $M$.  Consequently the cone
$\Strong^1(V)$ is generated by elements of the form $i\ell \wedge \overline \ell$, where $\ell \in M$. So (*) will follow if we show that $i \ell \wedge \overline{\ell} \in \Psef^1(B)$ for any $\ell \in M$.

Suppose to this end that
$\ell = s_1dz_1+ \dots + s_ndz_n \in M$, where $s_i \in S$. Consider the endomorphism
$$ \alpha : E^{\times n} \lra  E^{\times n}\ \ , \ \ \alpha(x_1, \ldots, x_n) \ = \ (s_1 x_1 , \ldots, s_n x_n).$$ Composing with the map $E^{\times n} \lra E$ given by summation, we arrive at a morphism
\[ \beta : E^{\times n} \lra E \ \ \text{with}  \ \ \beta^*( dz ) \ = \ \ell. \]
Now $D =_{\text{def}}\beta^{-1} (0)$ is an effective divisor in $B$. On the other hand, the cohomology class of $[0]$ is given by a positive
real scalar multiple of $idz \wedge d\overline z$, and hence the  cohomology class of $D$ is a positive multiple of $i\ell \wedge \overline \ell$.
Thus $i\ell \wedge \overline \ell$ represents a pseudoeffective class, and we are done.
\end{proof}


  \section{Canonical cycles on the self-product of an \av}\label{axa}

  In this section we begin our investigation of  cycles on the product of a higher-dimensional principally polarized abelian variety with itself. In order to obtain uniform statements, we will work always with the algebra generated by some natural divisor classes on this product. We introduce and study these here, and check that for a very general abelian variety they span the whole numerical equivalence ring.

  We start with some notation. Let $(A,\theta)$ be a \ppav\ and write
$ p_1, p_2 : A\times A\lra A
$ for  the two projections. In $H^2(A\times A ,\Q)$, we consider the three classes
\[ \theta_1=p_1^*\theta\ \ , \ \ \theta_2=p_2^*\theta \ \ , \ \ \lambda = c_1(\cP), \]
where $\cP$ is the  Poincar\'e bundle. We denote by $N^\bullet_{\rm can}(A\times A)_\Q$ the subalgebra of $H^\bullet (A\times A,\Q)$ generated by these classes, and we set
$$N^\bullet_{\rm can}(A\times A)\ =\ N^\bullet_{\rm can}(A\times A)_\Q\otimes_\Q\R\ \subseteq \ N^\bullet(A\times A)\ \subseteq \  H^\bullet (A\times A,\R).
$$

\subsection*{Hodge classes on the self-product of a very general \av}

  When $(A,\theta)$ is very general, results of Tankeev and Ribet imply that the rational canonical classes on $A\times A$ are exactly the rational Hodge classes.

  \begin{proposition}[Tankeev, Ribet]\label{tank}
Let $(A,\theta)$ be a very general \ppav. Then
\[ N^k_{\rm can}(A\times A)_\Q\ =\ H^{k,k}(A\times A)\, \cap \, H^{2k}(A\times A,\Q)\]
for all integers $k$. In particular,
$$N^k_{\rm can}(A\times A)\ = \ N^k(A\times A).$$
\end{proposition}

  \begin{proof} By a result of Tankeev (\cite{tan}; see also \cite{rib}), the algebra of Hodge classes on $A\times A$ is generated by the Hodge classes of type $(1,1)$. So we only need to show that classes of divisors are spanned by $\theta_1 $, $\theta_2 $,
and  $\lambda$.
To this end, let  $D\subset A\times A$ be a prime divisor that dominates $A$ via the first projection. The cohomology class of the general fiber of $p_1:D\to A$ is constant, so we get a map $A\lra \Pic^0(A)\isom A$ mapping $a\in A$ to the class of $\OO_A(D_a-D_0)$. Since $\End (A)\isom \Z$, this map is multiplication by an integer $n$. The restriction of $\cO_A(D-p_2^*D_0)\otimes
\cP^{-n}$ to a general fiber of $p_1$ is then trivial, hence this line bundle must
    be  a pull-back via $p_1$.
   \end{proof}

 \subsection*{The $\GL_2(\R)$-action}\label{act}


Let    $M=\begin{pmatrix}a&b\\c&d\end{pmatrix} \in\cM_2(\Z)$ be an integer matrix. We associate to $M$ the endomorphism $$
 u_M: A\times A \lra  A\times A\ \ , \ \
 (x,y) \mapsto  (ax+by,cx + dy),
$$
i.e., $u_M(x,y)    = \ (x,y)\, {}^t\!M$.
This induces   actions
of $\GL_2(\R)$ on $N^\bullet_{\rm can}(A\times A)$, $N^\bullet(A\times A)$, and $H^\bullet(A\times A,\R)$.

Note that $\theta_1$, $\lambda$, and $\theta_2$ are each in one piece of the K\"unneth decomposition
\small
\begin{equation*}
H^2(A\times A,\R) \, \isom \,
\bigl( H^2(A,\R)\otimes H^0(A,\R)\bigr)\,
\oplus\,\bigl( H^1(A,\R)\otimes H^1(A,\R)\bigr) \,
 \oplus \, \bigl( H^0(A,\R)\otimes H^2(A,\R)\bigr),
\end{equation*}
\normalsize
 and  that $\begin{pmatrix}a&0\\0&1\end{pmatrix}$   acts by multiplication by $a^2$, $a$, and $1$ on the respective pieces.

 Moreover, the addition map
 $\sigma:A\times A\to A$ satisfies (\cite[p. 78]{mum}),
 $$\sigma^*\theta=\theta_1+\theta_2+\lambda,$$
 and this implies in turn that the involution $(x,y)\mapsto (y,x)$ swaps $\theta_1$ and $\theta_2$ and leaves $\lambda$ invariant.

  It follows that the representation of $\GL_2(\R)$ on $N^1_{\rm can}(A\times A)$ is isomorphic to $\Sym^2W$, where $W$ is the tautological $2$-dimensional representation. More precisely, if $(e_1,e_2)$ is a basis for $W$, the correspondence is:
 $$\theta_1 \leftrightarrow e_1^2\quad ,\quad \theta_2 \leftrightarrow e_2^2\quad ,\quad \lambda \leftrightarrow 2e_1e_2.
 $$
 In particular, with $M$ as above,  the matrix of $u_M^*$ in the basis $(\theta_1,\theta_2,\lambda)$ of $N^1_{\rm can}(A\times A)$  is
 \begin{equation}\label{ft}
 \begin{pmatrix}
a^2&c^2&2ac\\
 b^2&d^2&2bd\\
ab&cd&ad+bc
\end{pmatrix}.
 \end{equation}

\subsection*{The structure of the algebra of canonical classes}

We use this $\GL_2(\R)$-action to determine the structure of the algebra $N^\bullet_{\rm can}(A\times A)$.

 \begin{prop}\label{mu} Let $(A, \theta)$ be a  \ppav\ of dimension $g$. Set
 \[
 \mu\ = \ 4\theta_1\theta_2-\lambda^2\ \in \ N^2_{\rm can}(A\times A). \] Then for every $r\in\{0,\dots,g\}$, the  maps
 $$\Sym^r N^1_{\rm can}(A\times A) \lra N^r_{\rm can}(A\times A)
 $$
and
 $${}\cdot \mu^{g-r}: N^r_{\rm can}(A\times A)\lra N^{2g-r}_{\rm can}(A\times A)
 $$
are isomorphisms.
  \end{prop}


\begin{remark} \label{Perf.Pairing.Canon.Spaces}
We will see in the proof of the Proposition that the element $\mu^g $ generates $\Numm^{2g}(A \times A)$. It follows from this that the perfect pairing between $\Numm^k(A \times A)$ and $\Numm^{2g-k}(A \times A)$ restricts to a perfect pairing between the corresponding canonical subspaces.
\end{remark}

\begin{remark}\label{Identifying.Can.Cones}
The Proposition gives a concrete identification of the spaces $\Numm^k_{\rm can}(A \times A)$ as $(A, \theta)$ varies, via taking monomials in $\theta_1, \lambda$, and $\theta_2$ to be  a basis  of $\Numm^r(A \times A)$ for $r \le g$. Of course this is compatible with the identifications coming from the Gauss-Manin connection.
\end{remark}

 \begin{proof}[Proof of the Proposition]
 The  first  map
 is $\GL_2(\R)$-equivariant, and  by the definition of $\Numm^{\bullet}_{\rm can}(A \times A)$ it is surjective.  Note that $M\cdot\mu=(\det M)^2\mu$, hence the  decomposition
 $$\Sym^r (\Sym^2 W)\ \isom \  \bigoplus_{0\le 2i\le r}(\det W)^{\otimes 2i} \otimes\Sym^{2r-4i}W
 $$
 of $\GL_2(\R)$-modules  translates in our case into

\begin{equation}\label{iso}
\Sym^r N^1_{\rm can}(A\times A)
\  \isom\  \bigoplus_{0\le 2i\le r} \mu^i\cdot \Sym^{2r-4i}W.
 \end{equation}
Taking $r=2g$ in (\ref{iso}), we see that there is a unique one-dimensional piece on the right-hand-side. It must    correspond to $N^{2g}_{\rm can}(A\times A)$, which is therefore spanned by $\mu^g$. This proves both statements for $r=0$.

Similarly, taking $r=2g-1$  in (\ref{iso}), we see that there is a unique three-dimensional piece on the right-hand-side. It must   correspond to $N^{2g-1}_{\rm can}(A\times A)$, which is therefore isomorphic to  $\mu^{g-1}  \cdot N^1_{\rm can}(A\times A)$. This proves both statements for $r=1$.

  We proceed   by induction on $r$. Since, by induction, the factor $\mu^i\cdot \Sym^{2r-4-4i}W$ appears in $N^{r-2}_{\rm can}(A\times A)$ for all $0\le 2i\le r-2$, and multiplication by $ \mu^{g-r+1}$ is injective on that space,  $\bigoplus_{1\le 2i\le r} \mu^i\cdot \Sym^{2r-4i}W$ appears in $N^r_{\rm can}(A\times A)$ and multiplication by $ \mu^{g-r}$ is injective on that space. If $\Sym^{2r}W$ does not appear in $N^r_{\rm can}(A\times A)$, we have
  $$\dim(N^r_{\rm can}(A\times A))\ =\  \sum_{0\le 2i\le r-2} \binom{2r-4i-2}{2}
  $$
  whereas, by induction, $$\dim(N^{r-1}_{\rm can}(A\times A))\ = \ \sum_{0\le 2i\le r-1} \binom{2r-4i}{2}\ > \  \dim(N^r_{\rm can}(A\times A)).
  $$
  But this is incompatible with Lefschetz' theorem which says that since $2r\le  \frac12\dim (A\times A)$, multiplication by an ample class induces an injection of the former space in the latter. This proves the first statement for $r$.

  A similar argument gives
  $$ N^{2g-r}_{\rm can}(A\times A)
 \isom \bigoplus_{0\le 2i\le r} \mu^{g-r+i}\cdot \Sym^{2r-4i}W,
 $$
 which is the second statement. \end{proof}

 \begin{rema}
 The cohomology of $\cP$ is known (\cite[Corollary 1, p. 129]{mum}) and $\chi(A\times A,\cP)=(-1)^g$. Using Riemann-Roch, we get $\lambda^{2g}=(-1)^g(2g)!$.
 \end{rema}

  \begin{rema}\label{re16}
  Since
$M\cdot \theta_1$ is the pull-back of $\theta$ by the morphism $p_1\circ u_M: A\times A\to A$, we have
 \begin{equation}\label{rel}
 (a^2\theta_1+b^2\theta_2+ab\lambda)^{g+1}\ = \ 0
  \end{equation}
  for all $a$, $b\in \R$. In particular, by taking coefficients of the relevant monomials, one finds that
  \[ \theta_1^{g+1}\ =\ \theta_2^{g+1}\ =\ \theta_1^g\lambda\ =\ \theta_2^g\lambda\ =\ 0. \]
 \end{rema}


\subsection*{The canonical cones in $N^k_{\rm can}(A\times A)$}
We denote by
 $$\Peff^k_{\rm can}(A\times A)\ = \ \Peff^k(A\times A)\cap N^k_{\rm can}(A\times A)$$
 the cone of canonical pseudoeffective classes, with
 \[ \Stpo^k_{\rm can}(A\times A) \ , \ \Sepo^k_{\rm can}(A\times A)\ , \ \Weak^k_{\rm can}(A\times A)\ , \   \Nef^k_{\rm can}(A\times A)\]
 defined similarly. If  $u_M : A \times A \lra A \times A$ is an isogeny defined by an integer matrix $M \in \cM_2(\Z)$, then $u_M^*$ maps each of these cones to itself. Therefore each of these cones is stable under the action of $\GL_2(\RR)$ on the real vector space $N^k_{\rm can}(A\times A)$.
  For $(A,\theta)$ very general, these cones coincide with $\Peff^k(A\times A)$,  $\Stpo^k(A\times A)$, $\Sepo^k(A\times A)$, $\Pos^k(A\times A)$, and $\Nef^k(A\times A)$, respectively (Proposition \ref{tank}).

  \begin{remark}
Note that $\Nef^k_{\rm can}(A \times A)$ and $\Psef^{2g - k}_{\rm can}(A \times A)$ are cones respectively in the dual vector spaces $\Numm^k_{\rm can}(A \times A)$ and $\Numm^{2g - k}_{\rm can}(A \times A)$. It follows from the definitions that
\[ \Nef^k_{\rm can}(A\times A)\ \subseteq \ \Peff^{2g -k}_{\rm can}(A\times A)^\vee. \]
However we are unaware of any a priori reason that these must coincide (although this happens of course when  $(A,\theta)$ is very general).
  \end{remark}

 Finally, we note that the canonical subcones defined by positivity of forms are independent of $(A , \theta)$.

   \begin{proposition}
   \label{Constancy.Canon.Subcones}
Under the identifications described in Remark \ref{Identifying.Can.Cones}, the cones
\[ \Stpo^k_{\rm can}(A\times A) \ \ , \ \ \Sepo^k_{\rm can}(A\times A)\ \ , \ \ \Pos^k_{\rm can}(A\times A)\] do not depend on  $(A,\theta)$.
\end{proposition}

\begin{proof}
Write $A=U/\Lambda$. We may choose coordinates $(z_1,\dots,z_g,z_{g+1},\dots,z_{2g})$ on $V=U\oplus U$ such that
\begin{equation} \label{Choose.Coords}
\begin{aligned} 
  \theta_1 \ &= \ idz_1\wedge d\overline z_1\, + \, \cdots\, + \, idz_g\wedge d\overline z_g,\\
 \theta_2 \ &= \  idz_{g+1}\wedge d\overline z_{g+1}\, + \, \cdots\, + \, idz_{2g}\wedge d\overline z_{2g},\\
 \lambda \ &=  \ idz_1\wedge d\overline z_{g+1}\, + \, id z_{g+1}\wedge d\overline z_1
 \, + \, \cdots\, + \,
 idz_g\wedge d\overline z_{2g}\, + \, id  z_{2g}\wedge d \overline z_g.
  \end{aligned}
  \end{equation}
  Therefore the cones in question consist of the polynomials in these classes for which the corresponding forms (or their products with a power of $\mu$) satisfy the stated positivity conditions. The assertion follows.
 \end{proof}




\subsection*{Cycles of codimension  and dimension 1}\label{cc1}
We close this section by studying the canonical cones of curves and divisors. 

Keeping notation as above, the class $\theta_1$ is on the boundary of the cone $\Peff^1_{\rm can}(A\times A)=\Nef^1_{\rm can}(A\times A)$, hence so are the classes
 $$\theta_{a,b}\ =_{\text{def}}    \ M\cdot \theta_1\ =\ a^2\theta_1+b^2\theta_2+ab\lambda$$
 for all  $a$, $b\in\R$. These classes sweep out the boundary of the cone in question, and therefore:

 \begin{prop}\label{nef1}
 Let $(A,\theta)$ be a \ppav. We have
  \begin{align*}
  \Peff^1_{\rm can}(A\times A) \ &= \ \Nef^1_{\rm can}(A\times A)\ =\ \ \big \langle\GL_2(\R)\cdot \theta_1\big \rangle \\
  &=\ \big \{ a_1\theta_1+a_2\theta_2+a_3\lambda \mid  a_1\ge 0, a_2\ge 0, a_1a_2 \ge a_3^2 \big \}.     \end{align*}
\end{prop}

Turning to the case of $1$-cycles, we know from Proposition \ref{prop17} that the pseudoeffective cone is the closed convex cone generated by  the $(2g-1)$-fold cup self-products of elements of $ \Peff^1 (A\times A)$. We give here a direct computation of this cone. To this end, we
 begin with a few calculations. The nonzero linear form
$${}\cdot \mu^{g-1}: N^2_{\rm can}(A\times A)\lra N^{2g}_{\rm can}(A\times A)=\C\mu^g
$$
is  $\GL_2(\R)$-equivariant.
Apply a matrix $M\in \GL_2(\R)$ to $\mu^{g-1}\theta_1^2$. On the one hand, we get
$$
M \cdot(\mu^{g-1}\theta_1^2)\ =\ (\det M)^{2g}(\mu^{g-1}\theta_1^2),$$
 because this is how  $\GL_2(\R)$ acts on $\mu^g$. On the other hand, we have by (\ref{ft})
$$
M \cdot(\mu^{g-1}\theta_1^2)\ =( \det M)^{2g-2}\mu^{g-1}( a^2\theta_1+b^2\theta_2+ab\lambda)^2.$$
Expanding and comparing these two expressions, we find
  \begin{equation}\label{mumu}
\mu^{g-1}\theta_1^2\, =\, \mu^{g-1}\theta_2^2\, =\,\mu^{g-1}\theta_1\lambda\, =\,\mu^{g-1}\theta_2\lambda\, =\,\mu^{g-1}(\lambda^2+2\theta_1\theta_2)\, =\,0.
  \end{equation}

   \begin{proposition}
 Let $(A,\theta)$ be a \ppav\ of dimension $g$. Then  \begin{align*}
  \Peff^{2g-1}_{\rm can}(A\times A)\ &= \  \Nef^{2g-1}_{\rm can}(A\times A)\\ &= \ \big \langle\GL_2(\R)\cdot\mu^{g-1}\theta_1\big \rangle\\
 &=\ \big \{  \mu^{g-1}(a_1\theta_1+a_2\theta_2+a_3 \lambda) \mid  a_1\ge 0, a_2\ge 0, a_1a_2 \ge a_3^2\big \}.
    \end{align*}
 \end{proposition}
\noi In other words, multiplication by $\mu^{g-1}$ induces a bijection
$$
{}\cdot\mu^{g-1}:   \Peff^1_{\rm can}(A\times A)\stackrel{\sim}{\lra}   \Peff^{2g-1}_{\rm can}(A\times A)
.$$

\begin{proof}
We   use the basis   $(\mu^{g-1}\theta_1,\mu^{g-1}\theta_2,\mu^{g-1}\lambda  )$ for $ N^{2g-1}_{\rm can}(A\times A)$ provided   by Proposition \ref{mu}.
A nef class
 \begin{equation}\label{alp}
 \alpha=\mu^{g-1}(a_1\theta_1+a_2\theta_2+a_3 \lambda)
 \end{equation}
must satisfy, for all $b\in\R$,
 \begin{eqnarray*}
0&\le&\Big(\begin{pmatrix}1&b\\0&1\end{pmatrix}\cdot\alpha\Big) \theta_1\\
   &=&b^{2g-2} \mu^{g-1} (a_1(\theta_1+b^2\theta_2+b\lambda)+a_2\theta_2+a_3 (2b\theta_2+\lambda))\theta_1\\
   &=&b^{2g-2} \mu^{g-1}\theta_1 \theta_2(a_1b^2+a_2 +2a_3b ).
  \end{eqnarray*}
  This implies
 \begin{equation}\label{alpp}
a_1\ge 0\ ,\quad  a_2\ge 0\ ,\quad   a_1a_2 \ge a_3^2.
 \end{equation}

For the converse, consider the effective class $\alpha=\theta_1^g\theta_2^{g-1}$ and write it as in (\ref{alp}). Since  $\alpha\theta_1=\alpha\lambda=0$ (Remark \ref{re16}), we obtain by (\ref{mumu}) $a_2=a_3=0$, hence
$\alpha $
 is a   multiple of $\mu^{g-1}\theta_1$. This must be a {\em positive} multiple by (\ref{alpp}), hence $\mu^{g-1}\theta_1$ is   effective.
 Then all classes
$$
  M\cdot (\mu^{g-1}\theta_1)= (\det M^{2g-2})\mu^{g-1}(a^2\theta_1+b^2\theta_2+ab\lambda)  $$
are effective and the proposition follows.
     \end{proof}





 \section{2-cycles on the self-product of an abelian surface}\label{pos}

 In this section, we study in more detail the case $g = 2$. Thus from now on
   $(A,\theta)$ is a principally polarized abelian surface, and we are interested in the canonical 2-cycles on $A\times A$.

   Recall that $(\theta_1^2, \theta_1\theta_2,\theta_2^2, \theta_1\lambda,\theta_2\lambda,\lambda^2)$ is a basis for $N^2_{\rm can}(A\times A)$ (Proposition \ref{mu}).
 The relations given by (\ref{rel}) are
$$
0\, = \, \theta_1^3 \, = \, \theta_2^3\, = \, \theta_1^2\lambda \, = \,\theta_2^2\lambda \, = \,\theta_1\theta_2^2+\theta_2\lambda^2 \, = \,\theta_1^2\theta_2+ \theta_1\lambda^2 \, = \,6 \theta_1  \theta_2 \lambda+\lambda^3
$$
and  the only nonzero products of four classes among $\theta_1$, $\theta_2$, and $\lambda$ are
$$\theta_1^2\theta_2^2 \ = \  4\ \  ,\ \   \theta_1\theta_2\lambda^2\  = \  -4\ \ ,\ \ \lambda^4 \ = \ 24.
$$

 \subsection*{The canonical semipositive cone}\label{cc2}
   We endow the vector space  $\bigwedge^2V$ with the coordinates
$$(z_1\wedge z_2,z_1\wedge z_3,z_1\wedge z_4,z_2\wedge z_3,z_2\wedge z_4,z_3\wedge z_4),$$
and we assume that $\theta_1, \theta_2, \lambda$ are given by the expressions appearing in equation \eqref{Choose.Coords}. Then
the   Hermitian forms on $\bigwedge^2V$ associated with various   classes in $N_{\rm can}^2(A\times A)$ have matrices:
$$
h_{\theta_1\theta_2}=
\begin{pmatrix}
0&0&0&0&0&0\\
0&1&0&0&0&0\\
0&0&1&0&0&0\\
0&0&0&1&0&0\\
0&0&0&0&1&0\\
0&0&0&0&0&0
\end{pmatrix}\quad,\quad
h_{\lambda^2}=
2\begin{pmatrix}
0&0&0&0&0&1\\
0&-1&0&0&0&0\\
0&0&0&-1&0&0\\
0&0&-1&0&0&0\\
0&0&0&0&-1&0\\
1&0&0&0&0&0
\end{pmatrix},
$$
$$
h_{\theta_1^2}=
2\begin{pmatrix}
1&0&0&0&0&0\\
0&0&0&0&0&0\\
0&0&0&0&0&0\\
0&0&0&0&0&0\\
0&0&0&0&0&0\\
0&0&0&0&0&0
\end{pmatrix}\quad,\quad
h_{\theta_2^2}=
2\begin{pmatrix}
0&0&0&0&0&0\\
0&0&0&0&0&0\\
0&0&0&0&0&0\\
0&0&0&0&0&0\\
0&0&0&0&0&0\\
0&0&0&0&0&1
\end{pmatrix},$$
$$
h_{\theta_1\lambda}= \begin{pmatrix}
0&0&1&-1&0&0\\
0&0&0&0&0&0\\
1&0&0&0&0&0\\
-1&0&0&0&0&0\\
0&0&0&0&0&0\\
0&0&0&0&0&0
\end{pmatrix}\quad,\quad
h_{\theta_2\lambda}=
\begin{pmatrix}
0&0&0&0&0&0\\
0&0&0&0&0&0\\
0&0&0&0&0&1\\
0&0&0&0&0&-1\\
0&0&0&0&0&0\\
0&0&1&-1&0&0
\end{pmatrix}.
$$
A class
\begin{equation}\label{xi}
\alpha \ =\ a_1\theta_1^2 \, + \, a_2 \theta_1\theta_2\, + \ a_3\theta_2^2\, + \ a_4 \theta_1\lambda \, + \ a_5\theta_2\lambda \, + \ a_6\lambda^2
\end{equation}
in $N^2_{\rm can}(A\times A)$ therefore  corresponds to
$$
h_\alpha=
\begin{pmatrix}
2a_1&0&a_4&-a_4&0&2a_6\\
0&a_2-2a_6&0&0&0&0\\
a_4&0&a_2&-2a_6&0&a_5\\
-a_4&0&-2a_6&a_2&0&-a_5\\
0&0&0&0&a_2-2a_6&0\\
2a_6&0&a_5&-a_5&0&2a_3
\end{pmatrix}.
$$
Note that this is the direct sum of a $4 \times 4$ matrix and a $2\times 2$ diagonal matrix. One then  checks easily that $h_\alpha$ is semipositive (\ie, $\alpha$ is  in $\Sdef^2_{\rm can}(A\times A)$) if and only if  the matrix
\begin{equation*}\label{sp}
\begin{pmatrix}
2a_1&a_4&-a_4&2a_6\\
a_4&a_2&-2a_6&a_5\\
-a_4&-2a_6&a_2&-a_5\\
2a_6&a_5&-a_5&2a_3
\end{pmatrix}\quad\hbox{is semipositive.}
\end{equation*}
After elementary row and column operations, one sees that this is equivalent to  $a_2 -2a_6 \ge 0$ and the condition 
  \begin{equation}\label{spp}
  \begin{pmatrix}
  a_1& a_4& a_6\\
  a_4& a_2+2a_6& a_5\\
  a_6& a_5& a_3
\end{pmatrix}\quad\hbox{is semipositive.}
\end{equation}
 This is equivalent  to the following
   inequalities  (nonnegativity of principal minors):
   \begin{subequations}
   \begin{eqnarray}
a_1,\ a_2,\ a_3  &\ge& 0 ,\label{a1} \\
 a_2 &\ge& 2| a_6| ,\label{a2}\\
 a_1(a_2+2a_6)  &\ge&  a_4^2
 ,\label{a3}\\
 a_3(a_2+2a_6)   &\ge& a_5^2
 ,\label{a4}\\
a_1a_3    &\ge& a_6^2,\label{a5}\\
(a_1a_3-a_6^2)(a_2+2a_6)+2a_4a_5a_6  &\ge&  a_3a_4^2+a_1 a_5^2  ,\label{a6}
 \end{eqnarray}
 \end{subequations}

We now come to our main result.

 \begin{theo}\label{mainth}
  Let $(A,\theta)$ be a \pps. Then
   $$ \Semi^2_{\rm can}(A\times A)\ =\ \Peff^2_{\rm can}(A\times A)\ = \ \Sym^2\Peff_{\rm can}^1(A\times A).$$
\end{theo}

\begin{proof} It is enough to prove the equality of the outer terms in the statement.
 The cone $\Sym^2\Peff_{\rm can}^1(A\times A)$ is the (closed) convex cone generated by $\theta_2^2$ and
 \begin{eqnarray*}
&& (\theta_1+a^2\theta_2+a\lambda)(\theta_1+b^2\theta_2+b\lambda)\\
&=&\theta_1^2+ (a^2+b^2)\theta_1\theta_2+ a^2b^2\theta_2^2+ (a+b)\theta_1\lambda+ab(a+b)\theta_2\lambda+ab\lambda^2 \end{eqnarray*}
for all $a$, $b\in\R$. Let $\cC\subset\R^6$ be   the closed convex cone generated by
  $(1,a^2+b^2,a^2b^2,a+b, ab(a+b),  ab )$, for $a$, $b\in\R$.

   For each $t\in[-1,1]$, we let $\cC_t$ be the closed convex subcone of $\cC$ generated by all vectors as above for which $b=ta$. It is contained in the hyperplane $x_6=\frac{ t}{1+t^2}x_2$ and we look at it as contained in $\R^5$ by dropping the last coordinate.

We now follow a classical argument in convexity theory (see \cite[IV.2]{bar}).    The dual    $\cC_t^\vee\subset \R^{5\vee}$ is   defined by the condition
$$\forall a \in\R\qquad
 y_1+y_2a^2(1+t^2)+y_3a^4t^2+y_4 a(1+t)+y_5a^3t(1+t)
 \ge 0.
 $$
As is well-known, this is equivalent to saying that this degree-4 polynomial in $a$ is a linear combination with positive coefficients of polynomials of the type $(z_1+z_2a+z_3a^2)^2$, where
  $(z_1,z_2,z_3)$ is in $\R^3$. This gives the following generators $(y_1,y_2,y_3,y_4,y_5)$ for  $\cC_t^\vee$:
\begin{eqnarray*}
y_1&=&z_1^2,\\
y_4 (1+t)&=&2z_1z_2 ,\\
y_2(1+t^2) &=&2z_1z_3+z_2^2 ,\\
y_5t(1+t) &=&2z_2z_3 ,\\
y_3t^2 &=&z_3^2,
\end{eqnarray*}
for $(z_1,z_2,z_3)\in\R^3$. In particular, a point $(x_1,x_2,x_3,x_4,x_5)$ is in the double dual $(\cC_t^\vee)^\vee$ if and only if, for all $(z_1,z_2,z_3)\in\R^3$, we have
\begin{eqnarray*}
0&\le&x_1y_1+x_2y_2+x_3y_3+x_4y_4+x_5y_5\\
 &=&x_1z_1^2+x_2\frac{2z_1z_3+z_2^2}{1+t^2} +x_3\frac{z_3^2}{t^2}+x_4\frac{2z_1z_2}{1+t} +x_5\frac{2z_2z_3}{t(1+t)}.
 \end{eqnarray*}
Since $(\cC_t^\vee)^\vee=\cC_t$, the  cone  $\cC_t$ is therefore defined by the condition:
    $$\begin{pmatrix}
  x_1&\frac{1}{1+t} x_4&\frac{1}{1+t^2} x_2\\
\frac{1}{1+t} x_4&\frac{1}{1+t^2} x_2&\frac{1}{t(1+t)} x_5\\
\frac{1}{1+t^2} x_2&\frac{1}{t(1+t)} x_5&\frac{1}{t^2} x_3
\end{pmatrix}\quad\hbox{is semipositive.}$$
This is in turn equivalent to:
    $$\begin{pmatrix}
  x_1& x_4&\frac{t}{1+t^2} x_2\\
  x_4&\frac{(1+t)^2}{1+t^2} x_2& x_5\\
\frac{t}{1+t^2} x_2& x_5& x_3
\end{pmatrix}\quad\hbox{is semipositive,}$$
 or:
    $$\begin{pmatrix}
  x_1& x_4& x_6\\
  x_4& x_2+2x_6& x_5\\
  x_6& x_5& x_3
\end{pmatrix}\quad\hbox{is semipositive.}$$
This proves that $\cC$ contains all vectors $(x_1,\dots,x_6)$ such that $ x_2\ge 2|x_6|$ which satisfy in addition the condition (\ref{spp}), hence all semipositive canonical classes.
\end{proof}

%

  \subsection*{The canonical nef cone}\label{eqnef} As above, let
\[ \alpha \  = \ a_1\theta_1^2 \, + \, a_2 \theta_1\theta_2\, + \ a_3\theta_2^2\, + \ a_4 \theta_1\lambda \, + \ a_5\theta_2\lambda \, + \ a_6\lambda^2\]
be a class in $\Numm^2_{\rm can}(A \times A)$. If $\alpha$ is nef, then   \[ \alpha\cdot\theta_{a,1}\cdot \theta_{b,1}\ \ge\ 0\]  for any $a , b \in \RR$ thanks to the fact that $\theta_{a,1} \cdot \theta_{b,1}$ is pseudoeffective.  This inequality    becomes
$$
a_3a^2b^2-a_5ab(a+b)+(a_2-a_6)(a^2+b^2)-(a_2-6a_6)ab
-a_4(a+b)+a_1
\ge 0
$$
 for all $a$ and $b$  in $\R$. 
Writing this as a quadratic polynomial in $a$, this is equivalent to the following inequalities, valid on $\Nef^2_{\rm can}(A\times A)$:
\begin{subequations}
 \begin{eqnarray}
 a_1,\  a_3&\ge&0,\label{e1}\\
a_2&\ge& a_6,\label{e3}\\
 4a_1(a_2-a_6)&\ge&a_4^2,\label{e4}\\
4a_3(a_2-a_6)&\ge& a_5^2,\label{e5}
 \end{eqnarray}
 and
 \begin{equation} \label{e6}
 (a_5b^2+(a_2-6a_6)b+a_4)^2\le 4(a_3b^2-a_5b+a_2-a_6)((a_2-a_6)b^2-a_4b+a_1)
 \end{equation}
 \end{subequations}
 for all $b\in\R$.\footnote{Observe that a real quadratic polynomial $\alpha x^2 + \beta x + \gamma$ is $\ge 0$ for all $x$ if and only if $\beta^2 - 4 \alpha \gamma \le 0$ and $\alpha  \ge 0$ or $\gamma \ge 0$.}

 By Theorem \ref{mainth}, the products $\theta_{a,1} \cdot \theta_{b,1}$ generate $\Psef^2_{\rm can}(A \times A)$. Therefore:
 \begin{proposition} The inequalities $($\ref{e1}$)$--$($\ref{e6}$)$ define the nef cone in $\Numm^2(A)$ when $(A,\theta)$ is very general. \qed
 \end{proposition}

 \begin{exam}\label{exnn}
 Let $(A, \theta)$ be a very general principally polarized abelian surface. The class $\mu_t=4\theta_1\theta_2+t\lambda^2$ is nef if and only if $-1\le t\le \frac32$
 (when $a_1=a_3=a_4=a_5=0$,  the inequalities (\ref{e1})--(\ref{e6}) reduce to $ -\frac14 a_2\le a_6\le \frac38 a_2$).
 \end{exam}

  \subsection*{The canonical weakly positive cone}

We found it harder to characterize weakly positive classes. However, it is possible to produce some interesting explicit  examples:

\begin{proposition}\label{thp}
Let $(A, \theta)$ be a  principally polarized abelian surface and let $t$ be a real number. On $A\times A$, the class \[ \mu_t\ = \ 4\theta_1\theta_2+t\lambda^2\] is {\em not}   semipositive for $t\ne 0$, but is  weakly positive if and only if $  | t|\le 1$.
In particular, there is a strict inclusion
$$\Semi^2_{\rm can}(A\times A)\ \subsetneq \ \Weak^2_{\rm can}(A\times A).$$
\end{proposition}

Thus there exist nef (integral) classes in $N^2(A\times A)$ which are not pseudoeffective. This is the case for the class   $\mu=\mu_{-1}$ defined in Proposition \ref{mu}. (Compare Corollary \ref{cor45}.)

\begin{proof}[Proof of Proposition] It is clear from \S\ref{cc2}  that for $t$ nonzero, the Hermitian matrix $h_{\mu_t}=4h_{\theta_1\theta_2}+th_{\lambda^2}$ is not semipositive,\footnote{For instance, equation \eqref{a5} is violated.} so  $\mu_t$ is not in $\Sdef^2(V)$.

For the second half of the statement, we check directly Definition \ref{def1} with
\begin{eqnarray*}
\ell_1&=&p_1dz_1+p_2dz_2+p_3dz_3+p_4dz_4,\\
\ell_2&=&q_1dz_1+q_2dz_2+q_3dz_3+q_4dz_4.
\end{eqnarray*}
For any indices $i$ and $j$, we set
$$c_{ij}=p_iq_j-p_jq_i.$$
We need to show that $\mu_t \wedge i\ell_1\wedge \overline \ell_1 \wedge i \ell_2\wedge \overline \ell_2 \ge 0$ 
for all $\ell_1$ and $\ell_2$, if $|t| \le 1$. 
Letting $\omega_0$ be the canonical $(4,4)$-form on $V$ which defines the orientation (see \S\ref{Pos.Classes.Abl.Var.Section}), we have
$$\theta_1\wedge\theta_2\wedge  i\ell_1\wedge \overline \ell_1 \wedge i \ell_2\wedge \overline \ell_2
\ =\ 
(|c_{24}|^2+|c_{23}|^2+|c_{14}|^2+|c_{13}|^2)\omega_0
$$and
\[\lambda\wedge\lambda\wedge  i\ell_1\wedge \overline \ell_1 \wedge i \ell_2\wedge \overline \ell_2 
\ = \ \bigl(-2|c_{24}|^2-2|c_{13}|^2+4\Re(c_{34}\overline{c_{12}})-4\Re(c_{23}\overline{c_{14}})\bigr)\omega_0.
\]
Using the identity
$$c_{ij}c_{kl}\ = \ c_{il}c_{kj}+c_{ik}c_{jl},
$$
we obtain
\begin{eqnarray}
2| \Re(c_{34}\overline{c_{12}})|
&\le&
2|c_{34}\overline{c_{12}}|\nonumber \\
&=&
2|c_{34}c_{12}|\nonumber \\
&\le&
2|c_{23}c_{14}|+ 2| c_{13}c_{24}| \nonumber\\
&\le&
|c_{23}|^2+|c_{14}|^2+ | c_{13}|^2+|c_{24}|^2 .\label{sim}
\end{eqnarray}
On the other hand, we have simply
$$
2|\Re(c_{23}\overline{c_{14}})|
\ \le\ |c_{23}|^2+|c_{14}|^2,
$$
so that
$$
\Big| \frac{1}{\omega_0}\lambda\wedge\lambda\wedge  i\ell_1\wedge \overline \ell_1 \wedge i \ell_2\wedge \overline \ell_2\Big|
\ \le \ 4(|c_{24}|^2+|c_{23}|^2+|c_{14}|^2+|c_{13}|^2).
$$
This proves that $\mu_t$ is in
$\Pos^2(V)$ for $  | t|\le 1$. This is the best possible bound because, taking $p_2=p_4=q_1=q_3=0$, $p_1=q_2=q_4=1$, and $p_3$ real, we obtain
\begin{equation}\label{0}
(4\theta_1\wedge\theta_2+t \lambda\wedge\lambda)\wedge  i\ell_1\wedge \overline \ell_1 \wedge i \ell_2\wedge \overline \ell_2 \ = \ 4(p_3^2+2tp_3+1)\omega_0,
\end{equation}
and for this form to be nonnegative for all $p_3$, we need  $  | t|\le 1$. This proves the theorem.\end{proof}

 \begin{coro}\label{cor44}
 Let $(A, \theta)$ be a very general  principally polarized abelian surface. There is a strict inclusion
$$\Pos^2(A\times A) \ \subsetneq \ \Nef^2(A\times A).$$
 \end{coro}

\begin{proof} This is because the class $\mu_\frac32 $  is nef (Example \ref{exnn}) but not weakly positive (Proposition \ref{thp}).
\end{proof}

\begin{coro}\label{cor45}
Let $(A, \theta)$ be a  principally polarized abelian surface. On $A\times A$, the classes $4\theta_1\theta_2+\lambda^2$ and $  2\theta_1^2+2\theta_2^2-\lambda^2$ are weakly positive, hence nef, but their product is $-8 $.
\end{coro}

\begin{proof}
With the notation above, the product of $  2\theta_1^2+2\theta_2^2-\lambda^2$ with a strongly positive class is
   \begin{eqnarray*}
    &&4|c_{34}|^2+4|c_{12}|^2+2|c_{24}|^2+2|c_{13}|^2-4\Re(c_{34}\overline{c_{12}})+4\Re(c_{23}\overline{c_{14}})\\
    &\ge&2|c_{34}|^2+2|c_{12}|^2+2|c_{24}|^2+2|c_{13}|^2+4\Re(c_{23}\overline{c_{14}})\\
    &\ge&0,
  \end{eqnarray*}
  by an inequality similar to (\ref{sim}). So this class is weakly positive. The computation of its product with $\mu_1$ is left to the reader.  \end{proof}


\section{Complements}

We give in this section some variants and generalizations of the results established above.

To begin with, it follows from Proposition \ref{Isogeny.Prop} that Theorems \ref{Ell.Curve.Intro} and \ref{TheoremAIntro} from the Introduction remain valid for abelian varieties isogeneous to those appearing in the statements. In other words:
\begin{corollary} \label{CM.Thm.Cor}

\begin{itemize}
\item[(i).]
Let $B = V/\Lambda$ be an abelian variety isogeneous to the $n$-fold self product of an elliptic curve with complex multiplication. Then
\begin{gather*}  \Psef^k(B) \,  = \,\Strong^k(B) \,  = \, \Strong^k(V) \\  \Nef^k(B) \, = \, \Weak^k(B) \, = \, \Weak^k(V)
 \end{gather*}
 for every $0 \le k \le n$.
\vskip 5pt
\item[(ii).] Let $A $ be a very general principally polarized abelian surface, and suppose that $B $ is isogeneous to $A \times A$. Then $\Peff^2(B) =\Sym^2\Peff^1(B)$, and
\[ \Peff^2(B) \, = \, \Spos^2(B) \, = \, \Sdef^2(B)\,  \subsetneqq \, \Pos^2(B) \, \subsetneqq  \ \Nef^2(B).  \ \ \ \qed \]
\end{itemize}
\end{corollary}

What's more interesting is that Theorem \ref{mainth} implies some statements for canonical cycles of codimension and dimension two on the self-product of a principally polarized abelian variety of arbitrary dimension.

 \begin{proposition} Let $A$ be a \ppav\ of dimension $g$. Then\begin{equation}\label{1e}
\Semi_{\rm can}^2(A\times A)\ =\ \Psef_{\rm can}^2(A\times A)\ =\ \Sym^2\Psef_{\rm can}^1(A\times A)
\end{equation}
and
\begin{equation}\label{2e}
\Strong_{\rm can}^{2g-2}(A\times A)\ =\  \Psef_{\rm can}^{2g-2}(A\times A) \ = \ \Sym^{2g-2}\Psef_{\rm can}^1(A\times A).
\end{equation}
\end{proposition}

\begin{proof} As before, it suffices to establish the equality of the outer terms in each of the displayed formulae. For this,  observe that by  Proposition \ref{Constancy.Canon.Subcones}  the cones in question do not depend on $(A,\theta)$. So we are reduced to proving the proposition for any one \ppav\ of dimension $g$. We choose  a product
$A_0\times B$, where $A_0$ is an abelian surface, $B$  is a  \ppav\ of dimension $g-2$, and $A_0 \times B$ has the product polarization.  Let $b$ be a point in $B$. Observing that the canonical classes on $(A_0 \times B) \times (A_0 \times B)$ restrict to those on $A_0 \times A_0 = (A_0 \times \{b\}) \times (A_0 \times \{ b \})$, it follows from Proposition \ref{nef1}   that the restriction map
$$N^1_{\rm can}(A_0\times B \times A_0\times B)\to N^1_{\rm can}\big((A_0\times \{b\}) \times (A_0\times \{b\})\big) \ = \ N^1_{\rm can}(A_0 \times A_0)$$
is an isomorphism which induces a bijection between pseudoeffective (or  strongly positive)
cones.

Any semipositive  codimension 2 class $\alpha$ restricts to a semipositive class on
$A_0 \times A_0$. By Theorem \ref{mainth}   applied to  $A_0$, it is a positive $\R$-linear combination of squares of strongly positive  of  codimension 1 classes. Since the restriction
$$N^2_{\rm can}(A_0\times B\times A_0\times B)\to N^2_{\rm can}(A_0\times A_0)$$
is an isomorphism, $\alpha$ itself is   a positive $\R$-linear combination of squares  of  codimension 1 classes, which are   strongly positive. This proves (\ref{1e}).

Now (\ref{2e})  is an immediate consequence of (\ref{1e}), using Proposition  \ref{pont} and Lemma \ref{prodform} (b). Indeed,
using the map $d\circ PD$ on $A\times A$ (and identifying $\widehat{A\times A}$ with
$A\times A$), we can reformulate (\ref{1e}) by saying that for  any \ppav\ $A$,
$$ \Strong_{2,{\rm can}}(A\times A)=\Psef_{2,{\rm can}}(A\times A)=\Sym^2\Psef_{1,{\rm can}}(A\times A),$$
where the product map $\Sym^2 N_1(A)\to N_2(A)$ is now given by the Pontryagin product.

This says that a pseudoeffective $2$-cycle class   is in the closed convex cone generated by the classes
$\gamma^{*2}$, where $\gamma\in \Strong_{1,{\rm can}}(A\times A)$. But, as we saw in the proof of Proposition \ref{prop17}, any class   in the interior of $ \Strong_{1,{\rm can}}(A\times A)$ can be written as $\gamma=\beta^{2g-1} $, where $\beta$ is an ample   divisor class.
 Lemma \ref{prodform} (a) then says that $\gamma^{*2}$ is a positive multiple of
 $\beta^{2g-2} $, hence is in $\Sym^{2g-2}\Psef_{\rm can}^1(A\times A)$. This  concludes the proof.
\end{proof}

Finally, we observe that Theorem \ref{CM}  implies that any strongly positive class on an abelian variety can be written as a limit of pseudoeffective cycles on small deformations of the given variety. 

\begin{coro} Let $B$ be an abelian variety and
$\alpha\in {\rm Strong}^{k}(B)$. Then there exist
a family of abelian varieties \[ p:\mathcal{B}\lra T,\] parameterized by a complex ball
$T$,  with
$B\cong \mathcal{B}_0$,  together with points $t_n\in T$ converging to $0$ in the classical topology, and
effective codimension $k$ $\QQ$-cycles
$  Z_n $   on $ \mathcal{B}_{t_n}$,
   such that
\begin{eqnarray}\label{limit}\lim_{n\rightarrow\infty}[Z_n]\ =\ \alpha.
\end{eqnarray}
\end{coro}
\noi The limit in (\ref{limit}) is taken in  the real vector space
$H^{2k}(B,\mathbf{R})$  which is canonically identified to $H^{2k}(\mathcal{B}_{t_n},\mathbf{R})$
for any $n$.

\begin{proof}
Indeed, we choose a polarization on $B$ and
we take for $\mathcal{B}\rightarrow T$ a  universal family of polarized deformations of $B$.
It is well known that points in $T$ parameterizing abelian varieties isogenous
to a self-product of an elliptic curve with complex multiplication are dense in $T$. Thus
we can choose $t_n\in T$ such that $\lim_{n\rightarrow\infty}t_n=0$ and $\mathcal{B}_{t_n}$ is isogenous
to a self-product of an elliptic curve with complex multiplication. We can then approximate
$\alpha$ by $\alpha_n\in  {\rm Strong}^{k}(\mathcal{B}_{t_n})$.  But then Corollary \ref{CM.Thm.Cor} (i)
   implies that the
$\alpha_n$'s  are pseudoeffective real classes
on $\mathcal{B}_{t_n}$ and can thus be approximated  by classes
of effective $\mathbf{Q}$-cycles on $\mathcal{B}_{t_n}$.
\end{proof}

 \section{Questions and conjectures}\label{quco}
 In this section we propose some questions and conjectures concerning positivity for cycles of codimension $> 1$.

 \subsection*{Abelian varieties}  There are a number of natural questions concerning positivity of cycles on abelian varieties and their products.

  To begin with, let $(A, \theta)$ be a principally polarized abelian variety. \begin{problem}
 Are the canonical cones $\Nef^k_{\rm can}(A \times A)$ and $\Psef^k_{\rm can}(A \times A)$ independent of $(A,\theta)$?
 \end{problem}
 Next let $B$ be an arbitrary abelian variety of dimension $n$, and let $\widehat B$ be its dual. As we saw in Proposition \ref{pont}, one has the Fourier-Mukai isomorphism
 \[N_k(B) \, = \,  N^{n-k}(B)\stackrel{{}_\sim}{\lra} N^k(\widehat{B}). \]
 \begin{conjecture}
 This isomorphism carries $\Psef_k(B)$ onto $\Psef^k(\widehat{B}),$ and $\Nef_k(B)$ onto $\Nef^k(\widehat{B})$. \end{conjecture}
\noi According to Proposition \ref{pont}, it interchanges in any event the strong cones of $B$ and $\widehat B$. It follows from this and the equalities appearing in equation
  \eqref {Cones.For.Divisors}
 that the conjecture is true when $k = 1$ and $k = n-1$.

 Finally, in all the examples considered above, it turned out that the pseudoeffective and strong cones coincided. This suggests:
 \begin{problem}
 Is $\Psef^k(B) = \Strong^k(B)$ for an arbitrary abelian variety $B$?
 \end{problem}
\noi We suspect that this is not the case, but it would be nice to find an actual example where it fails. We note that the equality of these cones on a given abelian variety $B$ implies the validity on that variety of a conjecture \cite{voi} of the fourth author giving a criterion for a class to lie in the interior of the pseudoeffective cone: see Remark \ref{Very.Moving.Abl.Var} below.

 \subsection*{Characterization of big classes}  Let $X$ be a smooth projective variety  of dimension $n$. By analogy with the case of divisors, one defines a class $\alpha \in \Numm^k(X)$ to be \textit{big} if it lies in the interior of $\Psef^k(X)$. There has been a certain amount of recent interest in the question of trying to recognize big cycles geometrically, the intuition being that they should be those that ``sit positively" in $X$, or that ``move a lot."

 This circle of thought started with Peternell \cite{pet}, who asked whether the cohomology class of a smooth subvariety $Y \subseteq X$ with ample normal bundle must be big. The fourth author gave a counter-example in \cite{voi}, involving a   codimension two subvariety $Y$ that actually moves in a family  covering $X$. The proof that
 \[ [Y] \ \in \  \text{Boundary} \big(\Psef^2(X) \big) \] revolves around the fact that $X$ carries a holomorphic form   vanishing on $Y$. This led her to conjecture that $[Y]$ will be big once  it is sufficiently mobile to rule out this sort of behavior. Specifically, one says that  $Y$ is ``very moving" in $X$ if roughly speaking it moves in a sufficiently large family so that given a general point $x \in X$, there are members of the family passing through $x$ whose tangent spaces dominate the appropriate Grassmannian of  $T_xX$. She proposes in \cite{voi} that this condition should guarantee the bigness of $[Y]$, and she proves that the truth of this conjecture would have striking Hodge-theoretic consequences for certain complete  intersections in projective space.

 \begin{remark} \label{Very.Moving.Abl.Var}
 Consider an abelian variety $B = V/\Lambda$ of dimension $n$ having the property that $$\Psef^k(B) \ = \  \Strong^k(B).$$ If $Y \subseteq B$ is a ``very moving" codimension $k$ subvariety of $B$, then $[Y] \in \textnormal{int}\big( \Psef^k(B)\big) $, i.e. the conjecture of \cite{voi} is true for $X = B$. In fact, suppose to the contrary that the class $[Y]$ of $Y$ lies on the boundary of  $\Psef^k(B) = \Strong^k(B)$. Then there would exist a weakly positive $(n-k, n-k)$ form $\eta$ on $V$ such that $\int_Y \eta = 0$. On the other hand, by the mobility hypothesis there is a deformation $Y^\pr$ of $Y$, and a smooth point $x \in Y^\pr$, at which $\eta|Y^\pr$ is strictly positive. Thus $\int_{Y^\pr} \eta > 0$, a contradiction.
\end{remark}

 It is immediate that $\alpha \in \Numm^k(X)$ is big if and only
 \[
( \alpha - \eps h^k) \ \in \ \Psef^k(X) \  \ \text{for} \  \ 0 < \eps \ll 1,
 \]
 where $h$ denotes the class of an ample divisor. This leads to the hope that it might be possible to characterize big classes as those that ``move as much" as complete intersection subvarieties.
 \begin{conjecture}
 A class $\alpha \in \Numm^k(X)$ is big if and only if the following condition holds:
 \begin{quote}
 There exist a constant $C > 0$, and arbitrary large integers $m$, with the property that one can find an  effective cycle $z_m$ in the class of $m \cdot \alpha$ passing through
 \[  \ge \ C \cdot m^{n/k} \]
 very general points of $X$.
 \end{quote}
 \end{conjecture}
\noi  We remark that the exponent $\frac{n}{k}$ appearing here is the largest that can occur. It is elementary that big classes do satisfy this condition, and that the conjecture holds in the classical cases $k= n-1$ and $k = 1$.  We have been able to verify the statement in one non-trivial case, namely when $k = 2$ and $\textnormal{Pic}(X) = \ZZ$.

More speculatively, if the conjecture (or something like it) is true, one is tempted to wonder whether one can measure asymptotically the ``mobility" of a cycle class to arrive at a continuous function
\[  \textnormal{mob}_X^k  : \Numm^k(X) \lra \RR \]
that is positive exactly on the big cone.\footnote{The idea would be to generalize the volume function $\textnormal{vol}_X : \Numm^1(X) \lra \RR$ that cuts out the cone of pseudoeffective divisors.} It would already be interesting to know if a natural function of this sort exists in the case $k = n-1$ of $1$-cycles.

\subsection*{Positivity of Chern and Schur classes} Let $X$ be a smooth projective variety of dimension $n$. It is natural to wonder about the positivity properties of the characteristic classes of positive vector bundles on $X$. Specifically, the results of  \cite{FL} (cf.\,\cite[Chapter 8]{PAG}) show that if $E$ is a nef vector bundle on $X$, then $c_k(E) \in \Nef^k(X)$, and more generally
\[
s_{\lambda}\big(c_1(E), \ldots , c_n(E) \big) \ \in \ \Nef^k(X),
\]
where $s_\lambda$ is the Schur polynomial associated to a partition $\lambda$ of $k$. Furthermore, the $s_\lambda$ span the cone of all weighted homogeneous polynomials $P$ of degree $k$ such that $P\big( c(E) \big) \in \Nef^k(X)$ whenever $E$ is a nef bundle on $X$.

However the effectivity properties of these classes is much less clear.
\begin{problem} \label{Which.Classes.Psef}
Characterize the cone $ \Sigma_{k,n}$ of all weighted homogeneous polynomials \[ P \ = \ P(c_1, \ldots, c_n)\]  of degree $k$ having the property that
\[  P\big(c_1(E), \ldots , c_n(E) \big) \ \in \ \Psef^k(X) \]
whenever $E$ is a nef vector bundle on any smooth projective variety $X$ of dimension $n$.
\end{problem}
\noi
Taking $X$ to be a Grassmannian, one sees that $\Sigma_{k,n}$ is contained in the cone spanned by the $s_\lambda$. On the other hand, $\Sigma_{k,n}$ contains the cone generated by the inverse Segre classes
\[ s_{(1^{\times \ell})}(E) \ \in \ \Numm^\ell(X)
\] 
and their products.  One could imagine that $\Sigma_{k,n}$ coincides with one of these two boundary possibilities.

Problem \ref{Which.Classes.Psef} is related to a problem of a combinatorial nature. Let $G = \textnormal{GL}_e( \CC)$ acting in the natural way on $V_m = \Sym^m \CC^e$, and let $Z = Z_m \subseteq V_m$ be an irreducible G-stable subvariety of codimension $k$. Then $Z$ determines an equivariant cohomology class
\[   [Z ]_G \ = \ H^{2k}_G( V_m) \ = \ \ZZ[c_1, \ldots, c_e]_{\text{deg} = k},  \]
and we consider the closed convex  cone:
 \[T_{k,e,m} \ \subseteq  \ \RR[c_1, \ldots, c_e]_{\text{deg} = k} \]
 generated by all such classes.
\begin{problem}
Compute the intersection \[T_{k,e}
\ = \ \bigcap_m \, T_{k,e,m} \ \subseteq \ \RR[c_1, \ldots, c_e]_{\textnormal{deg} = k}.\]
\end{problem}
\noi One can think of $T_{k,e}$ as a sort of equivariant pseudoeffective cone. The connection with Problem \ref{Which.Classes.Psef} is that if $e \ge n$, then  ${T}_{k,e} \subseteq \Sigma_{k,n}$.\footnote{If $Z = Z_m \subseteq \Sym^m \CC^e$ is $G$-equivariant, then $Z$ determines a cone $Z(E)$ of codimension $k$ inside the total space of $Sym^mE$ for any vector bundle of rank $e$, whose intersection with the zero section is the class determined from the Chern classes of $E$ by $[Z]_G$.   On the other hand, if $E$ is nef, and if $m$ is sufficiently large (depending on $E$), then the intersection of any effective cone in $\Sym^m E$ with the zero section is a pseudoeffective class. }

\subsection*{Special varieties}
It is natural to try to describe the higher codimension nef and pseudoeffective cones on special classes of varieties. As Diane Maclagen suggests, the toric case presents itself naturally here.
\begin{problem}
Interpret combinatorially the cones
\[  \Nef^k(X) \ \ , \ \ \Psef^k(X) \ \subseteq \ \Numm^k(X) \]
when $X$ is a smooth toric variety. Can it happen in this case that there are nef classes that fail to be pseudoeffective?
\end{problem}
\noi Spherical varieties might also be natural to consider. The case of projective bundles over curves was worked out by Fulger in \cite{ful}.

In the classical situations $k = n-1$ and $k = 1$, there are various classes of varieties for which the pseudoeffective and nef cones are polyhedral: besides the case of toric manifolds, this holds most notably for Fano varieties thanks to the minimal model program (\cite{BCHM}). It is natural to wonder when something similar happens for cycles of higher codimension:
\begin{problem}
Find natural classes of smooth projective varieties $X$ on which
\[
\Psef^k(X) \ \ , \ \ \Nef^k(X)  \ \subseteq \ \ \Numm^k(X) \]
are rational polyhedra.
\end{problem}
\noi Presumably this holds when $X$ is toric, but already when $X$ is Fano an example suggested by Tschinkel shows that it can fail:

\begin{example}[Tschinkel] Let $Y \subseteq \PP^4$ be a smooth surface, and let
\[  \nu : X = \Bl_Y(\PP^4) \lra \PP^4 \]
be the blowing up of $\PP^4$ along $Y$, with exceptional divisor $E$. Then
\[  \Numm^2(X) \ = \ \Numm^2(\PP^4) \, \oplus \, \Numm^1(Y) \ = \ \R \oplus \Numm^1(Y), \]
and one can show that $\Psef^1(Y) \subseteq \Numm^1(Y)$ appears on the boundary of $\Psef^2(X)$.\footnote{Geometrically, given an effective curve $\gamma \subseteq Y$, denote by $E_\gamma \subseteq X$ be the inverse image of $\gamma$ under the map $E \lra Y$. Then $E_{\gamma}$ is an effective surface on $X$ whose class lies on the boundary of $\Psef^2(X)$.} Now take $Y \subseteq \PP^3$ to be a quartic surface with the property that $\Psef^1(Y) = \Nef^1(Y)$ is a round cone: the existence of such K3's is verified for instance by Cutkosky in \cite{cut}. Viewing $Y$ as a subvariety of $\PP^4$ via a linear embedding $\PP^3 \subseteq \PP^4$, the resulting blow-up $X$ is Fano, and we get the required example. \qed
\end{example}

Finally, it might be interesting to study the interplay between classical geometry and the behavior of positive cones. For example, let $C$ be a smooth projective curve of genus $g \ge 2$, and consider the $n^{\text{th}}$ symmetric product $X = C_n= \textnormal{Sym}^n(C)$. It is interesting ask how the pseudoeffective and nef cones of $C_n$ depend on the geometry of $C$. Much as above, we consider the subalgebra $\Numm^{\bull}_{\rm can}(C_n)$ of $\Numm^{\bull}(X)$ generated by the classes
\[ x \, , \, \theta \ \in \  \Numm^1(C_n)\] of $\C_{n-1}$ and the pull-back of the theta divisor of $C$ under the Abel-Jacobi map $  C_n \lra {\rm Jac}(C)$, giving rise to cones
\[   \Psef^k_{\rm can}(C_n) \  , \  \Nef^k_{\rm can}(C_n) \ \subseteq \ \Numm^k_{\rm can}(C_n). \] In the classical case $k = 1$, these cones have been the object of a considerable amount of recent work \cite{Kou}, \cite{Pac}, \cite{Chan}, \cite{Mustopa}. Roughly speaking, the picture that emerges is that for small values of $n$, one can detect special properties of $C$ -- for example whether or not it is hyperelliptic -- by special special behavior of $\Nef^1_{\rm can}(C_n)$ or $\Psef^1_{\rm can}(C_n)$. However for fixed large $n$, the picture becomes uniform as $C$ varies. This suggests:

\begin{problem}
Given $n \gg 0$, do either of the cones  $\Nef^k_{\rm can}(\textnormal{Sym}^n(C))$ or $\Psef^k_{\rm can}(\textnormal{Sym}^n(C))$  vary non-trivially with $C$ for some value of $k$?  If so, can one detect whether $C$ is hyperelliptic -- or otherwise special -- from their geometry? \end{problem}

 \subsection*{Nef versus pseudoeffective} We have seen that there can exist nef classes that are not pseudoeffective. It is tempting to imagine that the presence of such classes is quite common, and that they should appear for instance on a sufficiently complicated blow-up of any smooth variety of dimension $\ge 4$. This motivates the
\begin{problem} Find other examples of smooth varieties $X$ that carry nef classes that are not pseudoeffective. For example,
can one find in any birational equivalence class of dimension $n \ge 4$ a smooth projective variety $X$   with the property  that for $2 \le k \le n-2$, 
$\Nef^k(X)$ is not contained in $\Psef^k(X)$?  \end{problem}
 \noi One of the difficulties here is that there doesn't seem to be any easy geometric way to exhibit nef cycles that are not pseudoeffective.

 \subsection*{Positivity for higher codimension cycles}   The failure of Grothendieck's conjecture (Remark \ref{Grothendieck.Questions}) indicates that nefness (as defined above) is probably not the right notion of positivity for higher codimension cycles. On the other hand, when $B$ is an abelian variety, one expects that any reasonable definition of positivity should lead to the cones $\Psef^k(B)$.
 This suggests:
 \begin{problem}
 Find a good notion of positivity for cycles on a smooth projective variety  $X$ that reduces to pseudoeffectivity when $X $ is an abelian variety.
 \end{problem}
\noi

\end{document}